\documentclass{article}
\usepackage{amsfonts}
\usepackage{amsmath}
\usepackage{amsthm}
\usepackage{amssymb}
\usepackage{hyperref}

\usepackage{tikz}

\newtheorem{theorem}{Theorem}
\newtheorem{proposition}{Proposition}
\newtheorem{claim}{Claim}
\newtheorem{lemma}{Lemma}
\newtheorem{corollary}{Corollary}
\newcommand{\Z}{\mathbb{Z}}
\newcommand{\R}{\mathbb{R}}

\begin{document}
\title{Finding Certain Arithmetic Progressions in 2-Coloured Cyclic Groups}
\author{Matei Mandache\footnote{Mathematical Institute, University of Oxford. E-mail: \href{mailto:matei.mandache@maths.ox.ac.uk}{matei.mandache@maths.ox.ac.uk}
}}
\maketitle
\begin{abstract}
We say a pair of integers $(a, b)$ is findable if the following is true. For any $\delta > 0$ there exists a $p_0$ such that for any  prime $p \ge p_0$ and any red-blue colouring of $\Z /p\Z$ in which each colour has density at least $\delta$, we can find an arithmetic progression of length $a+b$ inside $\Z/p\Z$ whose first $a$ elements are red and whose last $b$ elements are blue.

Szemer\' edi's Theorem on arithmetic progressions implies that $(0,k)$ and $(1,k)$ are findable for any $k$. We prove that $(2, k)$ is also findable for any $k$. However, the same is not true of $(3, k)$. Indeed, we give a construction showing that $(3, 30000)$ is not findable. We also show that $(14, 14)$ is not findable.
\end{abstract}
\section{Introduction}
In 1975, Szemer\'{e}di \cite{szemeredi75} proved the following famous theorem about arithmetic progressions:

\begin{theorem}\label{szap}
Let $k$ be an integer and let $\delta > 0$. There exists an $N_0 = N_0(k, \delta)$ such that if $N \ge N_0$ and $A \subset [N]$ with $|A| \ge \delta N$ then $A$ must contain a non-trivial arithmetic progression of length $k$.
\end{theorem}

Where we use the notation $[N]$ to denote the set $\{ 1, 2, \dots, N \}$.
Several generalisations of Thoerem~\ref{szap} have since been proven (see, for example, \cite{polynomialszemeredi} and \cite{multidimszemeredi}). In the current paper, we consider another way of trying to generalise this. If we consider the set $[N]$ as being coloured red and blue, with the set $A$ being red and $[N] \setminus A$ being blue, then Theorem~\ref{szap} tells us we can find red arithmetic progressions (under suitable size conditions). But what about arithmetic progressions which mix the two colours? For example, what if we want to find an arithmetic progression of length 4 which is coloured $( \text{red}, \text{red}, \text{blue}, \text{blue})$?

When considering this question, it is natural to work in cyclic groups modulo a prime, rather than in the set $[N]$, or in arbitrary cyclic groups. This is because in the latter two cases we may have some simple obstructions to finding such patterns. For example, one could colour all the even numbers red and all the odd numbers blue. These obstructions do not happen in $\Z/p\Z$, making the question more interesting.

For any integers $a, b \ge 0$, we define an $(a, b)$-arithmetic progression, or $(a, b)$-AP for short, to be an arithmetic progression of length $a+b$ with the first $a$ elements being red and the last $b$ elements being blue. Given a pair of non-negative integers $(a, b)$, we say that $(a, b)$ is findable if for all $\delta > 0$ there is a $p_0$ such that for any  prime $p \ge p_0$, any red-blue colouring of $\Z /p\Z$ with each colour class having size at least $\delta p$ contains an $(a,b)$-AP with common difference $d \neq 0$.

We conclude the introduction by making some basic remarks about which pairs are findable.
Findability is symmetric and monotone, in that if $(a, b)$ is findable, then so is $(b, a)$, and $(a', b')$ for any $a' \le a$ and $b' \le b$. Theorem~\ref{szap} still holds if we replace $[N]$ with $\Z/p\Z$ (and have a size condition on $p$), so $(k, 0)$ is findable for all $k$. In fact, Theorem~\ref{szap} also implies that $(k,1)$ is findable for all $k$. Once we have found a red arithmetic progression of length $k$ in $\Z/p\Z$, say $\{a, a+d, \dots, a+ (k-1)d \}$, let $t$ be the minimal non-negative integer such that $a+td$ is blue. This must exist, because as $d$ and $p$ are coprime, $a+td$ ranges over the whole of $\Z/p\Z$, and $\Z/p\Z$ is not all red. Also, $t \ge k$. So $\{a+(t-k)d, a+(t+1-k)d, \dots, a+(t-1)d, a+td \}$ is a $(k, 1)$-AP.

This prompts several questions: Is it true that $(a,b)$ is findable for any $a$ and $b$? For which $a$ is $(a, k)$ findable for all $k$? What is the set of findable pairs $(a,b)$?
The aim of this paper is to make progress on these questions. Not all pairs are findable: in Section $\ref{14vs14}$, we show that $(14,14)$ is not findable. Our construction showing $(14,14)$ is not findable is a fractal-type construction, based on replacement rules which are iterated ad infinitum. We note that this construction has some similarities to constructions used to solve problems about pattern avoiding words, such as those in \cite{ccss2014} and \cite{dekking79}, which are based upon iterating a replacement rule. However, the construction in our case is somewhat different. $(a,k)$ is findable for all $k$ if and only if $a \le 2$: in Section~\ref{2k}, we show that $(2,k)$ is findable for all $k$, and in Section~\ref{3vs30000} we show that $(3,30000)$ is not findable. This leaves open the question of findability for some pairs, but there are only finitely many such pairs. Perhaps the most interesting open question is that of whether $(3,3)$ is findable.

\section{Finding $(2, k)$}\label{2k}
Let $k$ be a positive integer. We aim to show $(2,k)$ is findable. Let us assume that $\Z/p\Z$ has been red-blue coloured with at least $\delta p$ in each colour class and that this colouring contains no $(2,k)$-AP. We aim to reach a contradiction, provided $p$ is sufficiently large.

Colourings of $\Z/p\Z$ induce colourings of $\Z$ via the quotient map $\Z \to \Z/p\Z$. The induced colouring is periodic mod $p$, and contains no $(2, k)$-AP. For the rest of this section, unless otherwise stated, we will work with the colouring of $\Z$ thus induced. We use $[x, y]$ to denote the set $\{ n \in \Z | x \le n \le y \}$. We define a \emph{gap} to be an arithmetic progression whose smallest and largest elements are both red, but all other elements are blue.

First, a few words about the strategy of the proof. We can use Szemer\'{e}di's Theorem to find large gaps. It turns out that, once we have found a gap, we can deduce quite a few things about the rest of the colouring by considering various arithmetic progressions of length $k+2$ and using the fact that they are not $(2,k)$-APs. Roughly speaking, our strategy is to show that if $[0,d]$ is a gap, then in fact all reds must occur at multiples of $d$. By combining this with periodicity mod $p$, we reach a contradiction.  Unfortunately this only works for some values of $d$, namely when $d$ is a sufficiently large prime. However, there are other tricks we can use when $d$ is not prime, to reduce to the case when $d$ is prime. Before we can get to that, however, we need to do some preparatory work.

We start off by proving the following lemma. 
\begin{lemma}\label{finiteness}
Let $m \ge 1$ and suppose $[0,d]$ is a gap. Then the number of reds in $[0, md]$ is bounded by a function depending on $m$ and $k$ only.
\end{lemma}
\begin{proof}
Let $c$ be the least integer such that $m \le \left(\frac{k+1}{k} \right)^c$. We will show that the number of reds in $[0,md]$ is at most $3 (k+1)^c -1$ by induction on $c$.

In the base case $c=0$, we have $m=1$ and the result is trivial. For the induction step, $c \ge 1$ and we may assume that $m = \left(\frac{k+1}{k} \right)^c$, since this makes $m$ bigger without increasing $c$. We know by the induction hypothesis that there are at most $3(k+1)^{c-1}-1$ reds in the interval $\left[0, \left(\frac{k+1}{k} \right)^{c-1} d \right]$. Let $r$ be the least red with $r > \left(\frac{k+1}{k} \right)^{c-1} d$. If $r \ge \left(\frac{k+1}{k} \right)^c d$, we are done.

Otherwise, suppose $s$ is red with  $r < s \le \left(\frac{k+1}{k} \right)^c d$. We aim to show that there are at most $3k(k+1)^{c-1}-k$ such reds. Consider the arithmetic progression of length $k+2$ with first term $s$ and common difference $r-s$. The first two terms, $s$ and $r$, are both red, and so it must contain another red. We pick a $t \in \{1, 2, \dots, k \}$ such that $r + t(r-s)$ is red. Since $t(s-r) \le k \left(\left(\frac{k+1}{k} \right)^c d -  \left(\frac{k+1}{k} \right)^{c-1} d \right) = \left(\frac{k+1}{k} \right)^{c-1} d\le r$, we have that $0 \le r + t(r-s) < r$ so $r + t(r-s)$ must be one of the reds in $\left[0, \left(\frac{k+1}{k} \right)^{c-1} d \right]$. So there are $k$ possible values for $t$ and at most $3(k+1)^{c-1}-1$ possible values for $r + t(r-s)$ . Since $t$ and $r+t(r-s)$ determine $s$, this means there are at most $3k(k+1)^{c-1}-k$ possible values for $s$.

We now do the final sum, by adding the numbers of reds in the interval which are less than $r$, equal to $r$ and greater than $r$. The total is at most $3(k+1)^{c-1}-1 + 1 + 3k(k+1)^{c-1}-k = 3(k+1)^c - k \le 3(k+1)^c-1$ as required.
\end{proof}
Lemma~\ref{finiteness} has the following corollary.
\begin{corollary}\label{gaps}
Let $c$ be a positive real. Then there exists an $N = N(c, k) \ge 1$ such that the following holds:
if $[m_1, m_2]$ and $[n_1, n_2]$ are gaps with $n_2 - n_1 \ge N(m_2-m_1)$, then the distance between the gaps must be at least $c(n_2-n_1)$
\end{corollary}
\begin{proof}
Without loss of generality, $m_1 \ge n_2$. By translating, we may also assume that $n_1 = 0$. We prove the contrapositive: assume that the gaps are less than $cn_2$ apart, and aim to prove that $(n_2 - n_1 )/(m_2-m_1)$ is bounded by a function of $c$ and $k$. We have $m_1 \le (c+1)n_2$ and therefore $m_2 \le (c+2)n_2$. We apply Lemma~\ref{finiteness} with $d=n_2$ and $m=c+2$ to deduce that there are at most $b=b(c, k)$ reds in the interval $[0,m_2]$. Let $r_1 \le r_2 \le \dots \le r_a$ be these reds. So we have $r_1 =0, r_2 = n_2, r_{a-1}=m_1, r_a = m_2$ and $a \le b$. For each $2 \le i \le a-1$, we must have $r_i-r_{i-1} \le k (r_{i+1} - r_i)$, since otherwise the arithmetic progression of length $k+2$ with first term $r_{i+1}$ and common difference $r_i-r_{i+1}$ would be a $(2, k)$-AP. By combining these inequalities for all $i$, we get that $r_1-r_0 \le k^{a-1} (r_a -r_{a-1})$, which means $n_2-n_1 \le k^{b-1} (m_2-m_1)$. So $(n_2 - n_1 )/(m_2-m_1) \le k^{b-1}$, as required.
\end{proof}

We use the notation $d[x, y]$ to denote the set $\{ dn| n \in [x, y] \}$. We define a \emph{table} to be a set of the form $d[-\ell , m]$ in which the following hold:
\begin{itemize}
\item $-\ell d$ and $md$ are both red
\item If $d[n_1, n_2] \subset d[-\ell ,m]$ is a gap, then $n_2-n_1 \le \frac{m+ \ell}{k+1} $
\end{itemize}
Note that $d[n_1, n_2]$ being a gap is not the same as $[dn_1, dn_2]$ being a gap.
A table is then called \emph{extendable} if, for any $\ell' \ge \ell$ and $m' \ge m$ with $-\ell'd$ and $md$ both red, $d[-\ell', m']$ is also a table.

The reason behind this definition is the following: once we have a table $d[-\ell, m]$, we can obtain information about the red points in $[-\ell d, md] \setminus d[-\ell, m]$ by repeatedly considering arithmetic progressions of length $k+2$ whose first element is a red in $d[-\ell, m]$ and whose second element is a red in $[-\ell d, md] \setminus d[-\ell, m]$. This will give us information about the reds in $[-\ell d, md] \setminus d[-\ell, m]$, and in some cases (when $d$ is a sufficiently large prime), we will see that in fact no such reds can exist. However, we first need to establish some results concerning the existence of tables.

\begin{lemma}\label{tableexistence}
Suppose $[0,d]$ is a gap. There exists an $N_0= N_0(k)$ such that if $\ell, m \ge N_0$ and $-\ell d, md$ are both red, then $d[-\ell, m]$ is an extendable table.
\end{lemma}
\begin{proof}
We start off by showing that if the condition is satisfied for a suitable choice of $N_0$, then $d[-\ell, m]$ is a table. Let $d[n_1, n_2] \subset d[-\ell ,m]$ be a gap. By applying corollary~\ref{gaps} inside the set $d\Z$ with $n_1, n_2$ as given, $m_1 = 0$, $m_2 = 1$ and $c=k+1$, we get an $N=N(k)$ such that if $n_2-n_1 \ge N$ then either $n_2 \le -(k+1)(n_2-n_1)$ or $n_1 \ge (k+1)(n_2-n_1)+1$. In the former case we get that $\ell \ge -n_2 \ge (k+1)(n_2-n_1)$, and in the latter case we get that $m \ge n_1 \ge (k+1)(n_2-n_1)$, so either way $\ell + m \ge (k+1)(n_2-n_1)$ as required. We take $N_0(k) = (k+1)N(k)$. Then, if $n_2 - n_1 \ge N$ we are done by the above argument, while if $n_2 - n_1 < N$ we have $n_2 - n_1 < N_0/(k+1) < (\ell + m)/(k+1)$, so the condition holds either way. So $d[-\ell, m]$ is a table.

Now if $\ell' \ge \ell$ and $m'\ge m$ have $-\ell' d$ and $m'd$ both red, then $d[-\ell', m']$ also satisfies the conditions of the lemma, and so is also a table. Therefore $d[-\ell, m]$ is an extendable table.
\end{proof}
So, we can find extendable tables, provided we can find suitable reds $-\ell d$ and $md$. The following lemma helps us find the reds to use as $-\ell d$ and $md$. It is also useful because it implies that once we have an extendable table, the table can be extended to be arbitrarily large.
\begin{lemma}\label{edges}
Suppose $[0,d]$ is a gap. Every set of the form $d[n, (k+1)n]$ or $d[-(k+1)n, -n]$ contains a red element.
\end{lemma}
\begin{proof}
It is sufficient to find sequences of reds $m_1 d < m_2 d < \dots$ and $-\ell_1 d > -\ell_2 d > \dots$ with the following properties.
\begin{itemize}
\item $1 \le m_1, \ell_1 \le k+1$
\item For all $i \ge 1$, we have $m_{i+1} \le (k+1) m_i$ and $\ell_{i+1} \le (k+1) \ell_i$
\end{itemize}
This is sufficient because it ensures that each $d[n, (k+1)n]$ contains some $m_i d$ and each $d[-(k+1)n, -n]$ contains some $-\ell_i d$.

We set $m_1 = 1$. To find $\ell_1$, note that the arithmetic progression \[ (d, 0, -d, -2d, \dots, -kd) \] has the first two terms being red, and therefore must contain another red term somewhere. So there is a $\ell_1 \in \{1, 2, \dots, k\}$ with $-\ell_1 d$ red.

We define the rest inductively. Given $m_i$, consider the arithmetic progression $(0, m_id, 2m_id, \dots, (k+1)m_id)$. This has the first two terms red, so there must be a $2 \le t \le k+1$ with $tm_id$ red. So set $m_{i+1} = tm_i$. Likewise, given $\ell_i$, consider the arithmetic progression $(0, -\ell_id, -2 \ell_id, \dots, -(k+1)\ell_id)$. This has the first two terms red, so there must be a $2 \le t \le k+1$ with $-t \ell_id$ red. So set $\ell_{i+1} = t \ell_i$.
\end{proof}

\begin{proposition}\label{denominators}
Suppose $[0,d]$ is a gap, and $B \ge N_0(k)$ be an integer (where $N_0(k)$ is as in Lemma~\ref{tableexistence}). Then there exists a $D_0 = D_0(B,k)$ and an extendable table $d[-\ell, m]$ such that $\ell, m \ge B$ and all reds in $[-\ell d, md]$ occur at integers of the form $rd$, where $r \in \mathbb{Q}$ has denominator at most $D_0$. 
\end{proposition}

\begin{proof}
By Lemma~\ref{edges}, we can find $\ell, m \in [B, (k+1)B]$ such that $-\ell d$ and $m d$ are both red. The conditions of Lemma~\ref{tableexistence} apply, so $d[-\ell, m]$ is an extendable table.
Since $[0,d]$ is a gap, we can apply Lemma~\ref{finiteness} to deduce that there is some $R_0 = R_0(B, k)$ such that there are at most $R_0/2$ reds in $[0, ((k+1)B+1)d]$. Similarly, by reflecting in $d/2$ and applying the same result, there are at most $R_0/2$ reds in $[-(k+1)Bd, d]$. Putting this together, there are at most $R_0$ reds in $[-(k+1)Bd, (k+1)Bd]$ and hence at most $R_0$ reds in $[-\ell d, md]$.

Let $T$ be the set of reds in $d[-\ell, m]$ and let $S$ be the set of reds in $[-\ell d, md] \setminus d[-\ell,m]$. So $S$ and $T$ are disjoint and $S \cup T$ is the set of reds in $[-\ell d, md]$. We therefore have $|S \cup T| \le R_0$.

We construct a function $f:S \to S \cup T$ as follows. Let $s \in S$. Let $n_1 d$ be the greatest element of $T$ which is less than $s$ and $n_2 d$ be the least element of $T$ which is greater than $s$ (these must exist because $-\ell d, md \in T$). $d[n_1, n_2]$ is a gap, so by the table property of $d[-\ell, m]$ we have $n_2-n_1 \le \frac{m+ \ell}{k+1} $.

If $n_1 d + (k+1)(s-n_1 d) \le md$ then the arithmetic progression $(n_1 d, s, n_1 d +2(s-n_1d), \dots, n_1 d + (k+1)(s-n_1 d))$ stays inside $[-\ell d, md]$. This arithmetic progression has its first two terms being red, so some other term must be red. We pick a $t \in \{2, 3, \dots, k+1 \}$ such that $n_1 d + t(s-n_1 d)$ is red, and define $f(s)= n_1 d + t(s-n_1 d)$.

If however $n_1 d + (k+1)(s-n_1 d) > md$, we can subtract $k(n_2-n_1) d$ from both sides to get $n_2 d + (k+1)(s-n_2 d) > md - k(n_2 - n_1)d > (m-(k+1)(n_2-n_1))d \ge -\ell d$, where the last inequality holds because $(k+1)(n_2-n_1) \le \ell + m$. So the arithmetic progression $(n_2 d, s, n_2 d + 2(s-n_2 d), \dots , n_2 d + (k+1)(s-n_2 d))$ will stay inside $[-\ell d, md]$. The first two terms are red, so pick a $t \in \{2, 3, \dots, k+1 \}$ such that $n_2 d + t(s-n_2 d)$ is red, and define $f(s)= n_2 d + t(s-n_2 d)$.

Let $s \in S$ be arbitrary. We are interested in the sequence $s, f(s), f(f(s)), \dots$. Since $|S| \le R_0$, this sequence must either land in $T$ (in which case it terminates) or repeat itself within the first $R_0$ iterations. For each $s$, $f(s)$ is given by an equation of the form $f(s) = h d + ts$, where $h$ is some integer and $t \in \{2, 3, \dots, k+1 \}$ (both $h$ and $t$ depend on $s$). Composing a string of such equations, we can see that in general $f^x(s)$ has equation $f^x(s) = h' d + t_1t_2 \dots t_x s$ where $h'$ is some integer and all the $t_i$ lie in $\{2, 3, \dots, k+1 \}$. One of the following two cases must hold.
\begin{itemize}
\item[(1)] If $f^x(s) \in T$ for some $0 \le x \le R_0$, then we have that $h'd + t_1t_2 \dots t_x s$ is a multiple of $d$, and hence we can write $s = rd$ where $r$ is a rational whose denominator divides $t_1t_2 \dots t_x$. In particular, the denominator of $r$ is at most $(k+1)^{R_0}$
\item[(2)] If $f^x(s) = f^y(s)$ with $0 \le x < y \le R_0$ then we have $h'd + t_1t_2 \dots t_x s = h''d + t_1t_2 \dots t_y s$, which can be rearranged as $t_1t_2 \dots t_x(t_{x+1}t_{x+2} \dots t_y - 1)s = (h'-h'')d$. Hence we can write $s=rd$ where $r$ is a rational whose denominator divides $t_1t_2 \dots t_x(t_{x+1}t_{x+2} \dots t_y - 1)$. Again this implies the denominator is at most $(k+1)^{R_0}$
\end{itemize}
So we have shown that every red in $[-\ell d, md]$ can be written as $rd$ where $r \in \mathbb{Q}$ has denominator at most $(k+1)^{R_0}$. Therefore the result holds with $D_0(B,k) = (k+1)^{R_0(B, k)}$
\end{proof}

We define $p_0 = p_0(k)$ to be the least prime greater than $D_0(N_0(k),k)$.

\begin{claim}\label{divisibility}
If $[0,d]$ is a gap  and $d \ge p_0$ is prime, then all reds occur at multiples of $d$.
\end{claim}
\begin{proof}
We apply Proposition~\ref{denominators} with $B = N_0(k)$ to get an extendable table $d[-\ell, m]$ such that all reds in $[-\ell d, md]$ occur at integers of the form $rd$ where $r \in \mathbb{Q}$ has denominator at most $D_0(N_0(k),k) < p_0$. Since $d$ is a prime greater than the denominator of $r$, this implies $r$ must be an integer, and hence all reds in $[-\ell d, md]$ occur at multiples of $d$.

It is sufficient to show that all reds in $[-\ell' d, m'd]$ occur at multiples of $d$, for any table $d[-\ell', m']$ that is an extension of the table $d[-\ell, m]$ because by Lemma~\ref{edges} we can find arbitrarily large such tables. We will prove this by induction, extending the table one bit at a time. The table $d[-\ell, m]$ serves as a base case for our induction. 

Induction step: Suppose $d[-\ell', m']$ is a table such that the only reds in $[-\ell' d, m'd]$ occur at multiples of $d$. Let $m''d$ be the smallest multiple of $d$ greater than $m'd$ which is red. This exists by Lemma~\ref{edges}. If there are no reds apart from $m'd$ and $m''d$ in $[m'd, m''d]$, then we are done. Otherwise, let $s$ be the least red greater than $m'd$ ($s$ will not be a multiple of $d$). Since $[-\ell' d, m''d]$ is a table and $d[m' , m'']$ is a gap, we have $m''d-s < m''d-m'd \le \frac{m''d+\ell' d}{k+1}$. This implies that all the terms in the arithmetic progression $(m''d, s, m''d -2(m''d-s), \dots, m''d-(k+1)(m''d-s))$ are in $[-\ell'd, m''d]$. The first two terms $m''d$ and $s$ are both red. Therefore there is a $t \in \{2, 3, \dots, (k+1) \}$ such that $m''d - t (m''d - s)$ is red. $m''d - t (m''d-s)$ must be a multiple of $d$, because of the induction hypothesis and the minimality of $s$. Therefore $ts$ is a multiple of $d$. Since $d$ is a prime larger than $t$, we deduce that $s$ is a multiple of $d$, a contradiction.

This completes the induction step for $m'$. Mirroring the same argument tells us we can increment $\ell'$ as well, completing the proof.
\end{proof}

\begin{claim}\label{smallprimes}
Let $p$ be a prime. There exists a positive integer $n_p$ such that if $[0,d]$ is a gap with $p^{n_p} | d$, then we can find an arithmetic progression of length at least $d$ and common difference $p$ all of whose elements are blue.
\end{claim}
\begin{proof}
We apply Proposition~\ref{denominators} with $B = \max(p, N_0(k))$ to get a table $d[-\ell, m]$ with $\ell, m \ge p$ such that all reds in $[-\ell d, md]$ have the form $rd$ where $r \in \mathbb{Q}$ has denominator at most $D_0(B, k)$. Choose $n_p$ such that $p^{n_p} > D_0(B, k)$. We take our arithmetic progression to be the set of elements of $[-\ell d, md]$ that are congruent to 1 mod $p$. This is an arithmetic progression with common difference $p$ of length at least $\frac{d(\ell + m)}{p} \ge d$.

If $p^{n_p} | d$ then all the elements of this progression are blue. Indeed, if one was red, we could write it as $rd$ with $r \in \mathbb{Q}$ having denominator less than $p^{n_p}$. But since $p^{n_p} | d$ and $p \nmid rd$, $p^{n_p}$ must divide the denominator of $r$, a contradiction.
\end{proof}

We are finally in a position to prove that $(2, k)$ is findable. In fact, we prove this slightly stronger result:

\begin{theorem}\label{finally!}
Let $\delta > 0$ and let $k$ be a positive integer. There is a $q_0=q_0(k, \delta)$ such that whenever $q \ge q_0$ is prime and $\Z/q\Z$ is coloured red and blue such that at least $\delta q$ of the elements are blue and at least two of the elements are red, there is a $(2, k)$-AP.
\end{theorem}
\begin{proof}
By Szemer\' edi's theorem, we can pick $q_0$ such that we are guaranteed to find an arithmetic progression of length at least
\[ \prod_{\substack{ p \: \text{prime} \\ p < p_0}} p^{n_p} \]
all of whose elements are blue, where $p_0$ is as in Claim~\ref{divisibility} and $n_p$ is as in Claim~\ref{smallprimes}.

Let $(a_1, a_2, \dots, a_d-1)$ be an arithmetic progression of maximal length, subject to all the elements being blue. By applying an automorphism of $\Z/q\Z$ we can arrange for this progression to be $(1, 2, \dots, d-1)$. By maximality, $0$ and $d$ must both be red. Since there are at least two red elements, we have $d<q$

Let us assume there is no $(2, k)$-AP. If there is some prime $p < p_0$ such that $p^{n_p} | d$, we may apply Claim~\ref{smallprimes}, working in the colouring induced in $\Z$, to find a blue arithmetic progression of common difference $p$ and length at least $d$. $p \neq 0$ in $\Z/q\Z$ since $q \ge p_0$, and so the arithmetic progression descends to a valid arithmetic progression in $\Z/q\Z$. This contradicts the maximality of $(a_1, a_2, \dots, a_d-1)$.

If, on the other hand, there is no prime $p < p_0$ such that $p^{n_p} | d$, then consider the prime factorisation of $d$. Since
\[ d > \prod_{\substack{ p \: \text{prime} \\ p < p_0}} p^{n_p} \]
there must be some prime $p \ge p_0$ such that $p|d$. We manipulate the colouring of $\Z/q\Z$ by applying the automorphism
\begin{align*}
\phi: \Z/q\Z & \to \Z/q\Z \\
x & \mapsto xd/p 
\end{align*}
Then we will have that $0$ and $p$ are red while $[1, 2, \dots, p-1]$ are blue. In the colouring induced on $\Z$, $[0, p]$ is a gap and $p \ge p_0$ is prime, so we may apply Claim~\ref{divisibility} to deduce that all reds in the colouring occur at multiples of $p$. $q \in \Z$ must be red since it corresponds to $0 \in \Z/q\Z$, therefore $p|q$. Since $q$ is prime, we have $p=q$, but $p \le d < q$, a contradiction.
\end{proof}
\section{Construction for $(14, 14)$}\label{14vs14}
We aim to produce colourings for $\Z/p\Z$ which contain no $(14, 14)$-AP. To do this, it suffices to give such a colouring of $\R/\Z$, since this induces colourings of $\Z/p\Z$ for all $p$ via the map
\begin{align*}
\phi : \Z/p\Z & \to \R/\Z \\
x & \mapsto \frac{x}{p}
\end{align*}
and the induced colourings contain $(14,14)$-APs only if the original colouring contained $(14,14)$-APs. The majority of this section is devoted to defining a red-blue colouring of $\R/\Z$ and proving that this colouring has no $(14,14)$-AP.

\subsection{Fractal construction}\label{fc1}
We define a \emph{configuration} to be a partial colouring of $\R/\Z$ together with a partition of the uncoloured subset into intervals, and an assignment of a direction to each of these intervals. Each interval can either be assigned direction $+1$, i.e. ``forwards'', or direction $-1$, i.e. ``backwards''. We proceed to define a fractal by specifying an initial configuration $c_0$ and a replacement rule telling us how to go from $c_i$ to $c_{i+1}$. Note that, as a set, $\R/\Z$ may be identified with $[0,1)$. We will often do this for convenience.

Let $c_0$ be the following configuration:
\begin{itemize}
\item $[0,1/4]$ coloured red
\item $(1/4, 1/2)$ uncoloured interval direction $+1$
\item $[1/2, 3/4]$ coloured blue
\item $(3/4, 1)$ uncoloured interval direction $-1$
\end{itemize}
This is shown in Figure~\ref{c0}. We now define a replacement rule for directed intervals. Let $(a,b)$ be one of the uncoloured intervals, and $\epsilon$ be its direction. We define the function $f:[0,1] \to [a,b]$ as follows:
\[ f : x \mapsto \begin{cases}
    a + (b-a)x  & \quad \text{if } \epsilon = +1 \\
    b - (b-a)x  & \quad \text{if } \epsilon = -1 
\end{cases} \]
We replace $(a,b) = f((0,1))$ with the following:

\begin{itemize}
\item $f((0,1/5))$ uncoloured, direction $\epsilon$
\item $f([1/5, 2/5])$ coloured blue
\item $f((2/5, 3/5))$ uncoloured, direction $-\epsilon$
\item $f([3/5, 4/5])$ coloured red
\item $f((4/5, 1))$ uncoloured, direction $\epsilon$.
\end{itemize}

\begin{figure}

\begin{tikzpicture}[scale=3]
  \draw[thick] (0,1) -- (4,1)
  node[pos=0, below]    {0}
  node[pos=0.25, below] {1/4}
  node[pos=0.5, below]  {1/2}
  node[pos=0.75, below] {3/4}
  node[pos=1, below]    {1};
  \draw[->, thick]  (1,1) -- (1.5,1);
  \draw[->, thick]  (4,1) -- (3.5,1);
  \draw[red, very thick]  (0,1)  -- (1,1) ;
  \draw[blue, very thick] (2,1) -- (3,1);
\end{tikzpicture}
\caption{The configuration $c_0$}
\label{c0}
\end{figure}
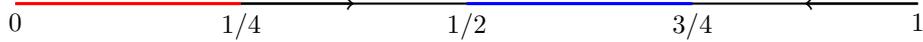

\begin{figure}
\begin{tikzpicture}[scale=1]
  \draw[thick] (0,0) -- (5,0)
  node[pos=0, below]    {0}
  node[pos=1, below]    {1};
  \draw[thick] (7,0) -- (12,0)
  node[pos=0, below]    {0}
  node[pos=0.2, below] {1/5}
  node[pos=0.4, below]  {2/5}
  node[pos=0.6, below]  {3/5}
  node[pos=0.8, below] {4/5}
  node[pos=1, below]    {1};
  \draw[->, thick]  (0,0) -- (2.5,0);
  \draw[->, thick]  (7,0) -- (7.5,0);
  \draw[->, thick]  (10,0) -- (9.5,0);
  \draw[->, thick]  (11,0) -- (11.5,0);
  \draw[red, very thick]  (10,0)  -- (11,0) ;
  \draw[blue, very thick] (8,0) -- (9,0);
  \draw[->] (5.8,0) -- (6,0) node[below] {Becomes} -- (6.2,0);
\end{tikzpicture}
\caption{The replacement rule for a directed interval. This is shown for the interval $(0,1)$. The rule for a general interval is obtained by applying the linear function $f$.}
\label{replacement} 
\end{figure}
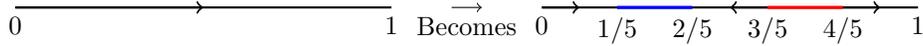

This is shown in Figure~\ref{replacement}. We can construct configurations $(c_i)_{i=0}^\infty$ inductively. For each $i \ge 0$, define $c_{i+1}$ to be the configuration obtained by applying the replacement rule to all the uncoloured intervals in $c_i$. Figure~\ref{firsttwo} shows $c_1$ and $c_2$. These colourings have the property that if a point receives a given colour in $c_i$, it will also receive that colour in $c_j$ for all $j \ge i$. Therefore we can define a partial colouring $c_\infty$, by colouring each point red if there is some $i$ such that $c_i$ colours that point red, and colouring it blue if there is some $i$ such that $c_i$ colours that point blue. This partial colouring $c_\infty$ will be a fractal. There will be some points (for example, the point $3/8$) that will not receive either colour in $c_\infty$, so this is not yet a full colouring of $\R/\Z$. We finally construct a colouring $c'_\infty$ by taking $c_\infty$ and colouring all the uncoloured points in some arbitrary way (for example, we could colour all the uncoloured points blue). This $c'_\infty$ is our counterexample colouring.

We now establish a symmetry of the configurations $c_i$, and use it to deduce a symmetry property of $c_\infty$. Let $T$ denote the following operation for a configuration:
\begin{itemize}
\item Translate everything by 1/2
\item For all coloured points, switch the colour between red and blue
\item For all uncoloured intervals, reverse the direction of the interval
\end{itemize}
One can see by inspection that $T(c_0) = c_0$. Also, by inspecting the replacement rule, we also have that, for any $i$, if $T(c_i)=c_i$, then $T(c_{i+1})=c_{i+1}$. Therefore, by induction, $T$ is a symmetry of $c_i$ for all $i$. By going to the limit, we can deduce that in $c_\infty$, $x$ is red if and only if $x+1/2$ is blue.

\begin{figure}
\begin{tikzpicture}[scale=3]
  \draw[thick] (0,0) -- (4,0)
  node[pos=0, below]    {0}
  node[pos=0.25, below] {1/4}
  node[pos=0.5, below]  {1/2}
  node[pos=0.75, below] {3/4}
  node[pos=1, below]    {1};
  \draw[->, thick]  (1,0) -- (1.1,0);
  \draw[->, thick]  (1.6,0) -- (1.5,0);
  \draw[->, thick]  (1.8,0) -- (1.9,0);
  \draw[->, thick]  (3.2,0) -- (3.1,0);
  \draw[->, thick]  (3.4,0) -- (3.5,0);
  \draw[->, thick]  (4,0) -- (3.9,0);
  \draw[red, very thick]  (0,0)  -- (1,0) ;
  \draw[red, very thick]  (1.6,0)  -- (1.8,0) ;
  \draw[red, very thick]  (3.2,0)  -- (3.4,0) ;
  \draw[blue, very thick] (2,0) -- (3,0);
  \draw[blue, very thick] (1.2,0) -- (1.4,0);
  \draw[blue, very thick] (3.6,0) -- (3.8,0);
\end{tikzpicture}

\begin{tikzpicture}[scale=3]
  \draw[thick] (0,0) -- (4,0)
  node[pos=0, below]    {0}
  node[pos=0.25, below] {1/4}
  node[pos=0.5, below]  {1/2}
  node[pos=0.75, below] {3/4}
  node[pos=1, below]    {1};
  \draw[->]  (1,0) -- (1.03,0);
  \draw[->]  (1.12,0) -- (1.09,0);
  \draw[->]  (1.16,0) -- (1.19,0);
  \draw[->]  (1.44,0) -- (1.41,0);
  \draw[->]  (1.48,0) -- (1.51,0);
  \draw[->]  (1.6,0) -- (1.57,0);
  \draw[->]  (1.8,0) -- (1.83,0);
  \draw[->]  (1.92,0) -- (1.89,0);
  \draw[->]  (1.96,0) -- (1.99,0);
  \draw[->]  (3.04,0) -- (3.01,0);
  \draw[->]  (3.08,0) -- (3.11,0);
  \draw[->]  (3.2,0) -- (3.17,0);
  \draw[->]  (3.4,0) -- (3.43,0);
  \draw[->]  (3.52,0) -- (3.49,0);
  \draw[->]  (3.56,0) -- (3.59,0);
  \draw[->]  (3.84,0) -- (3.81,0);
  \draw[->]  (3.88,0) -- (3.91,0);
  \draw[->]  (4,0) -- (3.97,0);
  \draw[red, very thick]  (0,0)  -- (1,0) ;
  \draw[red, very thick]  (1.6,0)  -- (1.8,0) ;
  \draw[red, very thick]  (3.2,0)  -- (3.4,0) ;
  \draw[red, very thick]  (1.12,0)  -- (1.16,0) ;
  \draw[red, very thick]  (1.44,0)  -- (1.48,0) ;
  \draw[red, very thick]  (1.92,0)  -- (1.96,0) ;
  \draw[red, very thick]  (3.04,0)  -- (3.08,0) ;
  \draw[red, very thick]  (3.52,0)  -- (3.56,0) ;
  \draw[red, very thick]  (3.84,0)  -- (3.88,0) ;
  \draw[blue, very thick] (2,0) -- (3,0);
  \draw[blue, very thick] (1.2,0) -- (1.4,0);
  \draw[blue, very thick] (1.04,0) -- (1.08,0);
  \draw[blue, very thick] (1.52,0) -- (1.56,0);
  \draw[blue, very thick] (1.84,0) -- (1.88,0);
  \draw[blue, very thick] (3.6,0) -- (3.8,0);
  \draw[blue, very thick] (3.12,0) -- (3.16,0);
  \draw[blue, very thick] (3.44,0) -- (3.48,0);
  \draw[blue, very thick] (3.92,0) -- (3.96,0);
\end{tikzpicture}
\caption{Above: the configuration $c_1$. Below: the configuration $c_2$}
\label{firsttwo}
\end{figure}
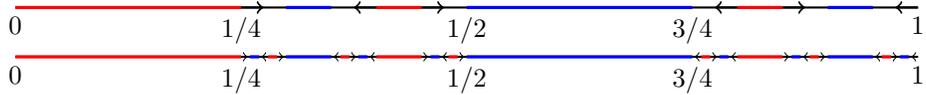
\subsection{Proof there is no $(14,14)$-AP}
When we have a partial colouring, we define an $(a,b)$-AP to be an arithmetic progression of length $a+b$ such that none of the first $a$ elements are blue and none of the last $b$ elements are red. When the partial colouring is in fact a full colouring, this notion coincides with the usual notion of an $(a,b)$-AP.

We start off by noting that no $(14,14)$-AP in $c'_\infty$ can have common difference 0 or 1/2, because in the former case all points have the same colour, and in the latter case the colouring has period 2. So it is sufficient to focus on the case where $d \in (-1/2,0) \cup (0,1/2)$. We will prove the following:
\begin{theorem}\label{cinfty}
$c_\infty$ does not contain any $(14,14)$-AP with common difference $d \in (-1/2, 0) \cup (0,1/2)$.
\end{theorem}
This implies that $c'_\infty$ does not contain any $(14,14)$-AP with $d \in (-1/2, 0) \cup (0,1/2)$.
\begin{proof}
Firstly, we observe that we may restrict attention to $d \in (0,1/2)$. Indeed, if $d \in (-1/2,0)$, then any $(14,14)$-AP $\{a, a+d, \dots , a+27d \}$ with common difference $d$ will give rise to another $(14,14)$-AP, $\{a+27d+1/2, a+26d + 1/2, \dots, a+d+1/2, a+1/2 \}$, with common difference $-d$. If the former is a $(14,14)$-AP then so is the latter because of the symmetry property of $c_\infty$.

We split into the following cases, according to where in the interval $(0,1/2)$ the common difference $d$ lies:
\begin{itemize}
\item[(1)] $d \in (\frac{1}{4 \cdot 5^{k+1}}, \frac{1}{4 \cdot 5^k}]$ for some integer $k \ge 0$
\item[(2)] $d \in [\frac{1}{4}, \frac{5}{16}]$
\item[(3)] $d \in [\frac{5}{16}, \frac{7}{20}]$
\item[(4)] $d \in [\frac{7}{20}, \frac{2}{5}]$
\item[(5)] $d \in [\frac{1}{2}-\frac{1}{8 \cdot 5^k}, \frac{1}{2}-\frac{1}{8 \cdot 5^{k+1}})$ for some integer $k \ge 0$.
\end{itemize}
Note that at least one of these cases applies. In each case we will show that, for some given $i$ and $a, b \le 14$, $c_i$ does not contain an $(a,b)$-AP with common difference $d$. This implies $c_\infty$ does not contain a $(14,14)$-AP with common difference $d$.

\textbf{Case (1)} Let $k \ge 0$ be an integer. We claim $c_{k+1}$ does not contain a $(14,14)$-AP with $d \in (\frac{1}{4 \cdot 5^{k+1}}, \frac{1}{4 \cdot 5^k}]$.

We start off by considering $c_k$. $c_k$ consists of the following repeating pattern:
\begin{itemize}
\item a blue interval of length at least $\frac{1}{4 \cdot 5^k}$
\item an uncoloured interval of length $\frac{1}{4 \cdot 5^k}$
\item a red interval of length at least $\frac{1}{4 \cdot 5^k}$
\item an uncoloured interval of length $\frac{1}{4 \cdot 5^k}$
\end{itemize}
This repeats itself over and over as we go around $\R/\Z$. The lengths of the red and blue intervals will vary as the pattern repeats itself. The pattern has the additional property that in any pair of adjacent coloured intervals, at least one of the coloured intervals has length exactly  $\frac{1}{4 \cdot 5^k}$. All these properties can easily be verified by induction on $k$.

We say an arithmetic progression jumps over an interval if it contains points before and after the interval, but does not contain any points in the interval. Obviously this can only happen if the length of the interval is smaller than the common difference. In this case, we assume $d \le \frac{1}{4 \cdot 5^k}$, so the arithmetic progression cannot jump over any of the intervals in $c_k$.

Let us assume for contradiction that $\{ a, a+d, a+2d, \dots, a+26d, a+27d \}$ is a $(14,14)$-AP. Since $a+13d$ is not blue and $a+14d$ is not red, the interval $[a+13d, a+14d]$ cannot be contained in a coloured interval. Therefore we can pick some $t \in [13,14]$ such that $a+td$ is in an uncoloured interval of $c_k$. Let this uncoloured interval be $(b, b+ \frac{1}{4 \cdot 5^k})$. Since $d > \frac{1}{4 \cdot 5^{k+1}}$, our arithmetic progression must contain points both before and after this uncoloured interval. Therefore, the uncoloured interval $(b, b+ \frac{1}{4 \cdot 5^k})$ is preceded by a red interval and succeeded by a blue interval. At least one of these two coloured intervals must have length exactly $\frac{1}{4 \cdot 5^k}$. Because the situation so far is symmetric, we may assume without loss of generality that the red interval has length $\frac{1}{4 \cdot 5^k}$. So we have the following situation, which is shown in Figure~\ref{aroundb}.
\begin{itemize}
\item $[b-\frac{3}{4 \cdot 5^k}, b-\frac{2}{4 \cdot 5^k}]$ is blue
\item $(b-\frac{2}{4 \cdot 5^k}, b-\frac{2}{4 \cdot 5^k})$ is uncoloured
\item $[b-\frac{1}{4 \cdot 5^k}, b]$ is red
\item $(b, b-\frac{2}{4 \cdot 5^k})$ is uncoloured
\item $[b+\frac{1}{4 \cdot 5^k}, b+\frac{2}{4 \cdot 5^k}]$ is blue
\end{itemize}
The first fourteen terms of the arithmetic progression are contained in the interval $(b- \frac{2}{4 \cdot 5^k}, b+ \frac{1}{4 \cdot 5^k})$.
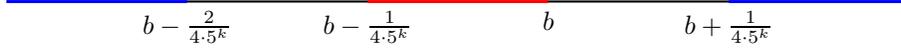
\begin{figure}
\begin{tikzpicture}[scale=2.4]
  \draw[thick] (0,1) -- (5,1)
  node[pos=0, below]    {}
  node[pos=0.2, below] {$b- \frac{2}{4 \cdot 5^k}$}
  node[pos=0.4, below]  {$b- \frac{1}{4 \cdot 5^k}$}
  node[pos=0.6, below]  {$b$}
  node[pos=0.8, below] {$b+ \frac{1}{4 \cdot 5^k}$}
  node[pos=1, below]    {};
  \draw[blue, very thick]  (0,1)  -- (1,1) ;
  \draw[red, very thick] (2,1) -- (3,1);
  \draw[blue, very thick]  (4,1)  -- (5,1) ;
\end{tikzpicture}
\caption{The configuration $c_k$ around the area of interest}
\label{aroundb}
\end{figure}
\begin{figure}
\begin{tikzpicture}[scale=2.4]
  \draw[thick] (0,1) -- (5,1)
  node[pos=0, below]    {}
  node[pos=0.2, below] {$b- \frac{2}{4 \cdot 5^k}$}
  node[pos=0.4, below]  {$b- \frac{1}{4 \cdot 5^k}$}
  node[pos=0.6, below]  {$b$}
  node[pos=0.8, below] {$b+ \frac{1}{4 \cdot 5^k}$}
  node[pos=1, below]    {};
  \draw[blue, very thick]  (1.6,1)  -- (1.8,1) ;
  \draw[red, very thick] (1.2,1) -- (1.4,1);
  \draw[blue, very thick]  (3.2,1)  -- (3.4,1) ;
  \draw[red, very thick] (3.6,1) -- (3.8,1);
  \draw[blue, very thick]  (0,1)  -- (1,1) ;
  \draw[red, very thick] (2,1) -- (3,1);
  \draw[blue, very thick]  (4,1)  -- (5,1) ;
  \draw (1, 0.95) -- (1, 1.05);
  \draw (1.2, 0.95) -- (1.2, 1.05);
  \draw (1.4, 0.95) -- (1.4, 1.05);
  \draw (1.6, 0.95) -- (1.6, 1.05);
  \draw (1.8, 0.95) -- (1.8, 1.05);
  \draw (2, 0.95) -- (2, 1.05);
  \draw (2.2, 0.95) -- (2.2, 1.05);
  \draw (2.4, 0.95) -- (2.4, 1.05);
  \draw (2.6, 0.95) -- (2.6, 1.05);
  \draw (2.8, 0.95) -- (2.8, 1.05);
  \draw (3, 0.95) -- (3, 1.05);
  \draw (3.2, 0.95) -- (3.2, 1.05);
  \draw (3.4, 0.95) -- (3.4, 1.05);
  \draw (3.6, 0.95) -- (3.6, 1.05);
  \draw (3.8, 0.95) -- (3.8, 1.05);
  \draw (4, 0.95) -- (4, 1.05);
\end{tikzpicture}
\caption{The configuration $c_{k+1}$ around the area of interest, and the partition of $(b- \frac{2}{4 \cdot 5^k}, b+ \frac{1}{4 \cdot 5^k})$ into fifteen intervals}
\label{aroundb2}
\end{figure}
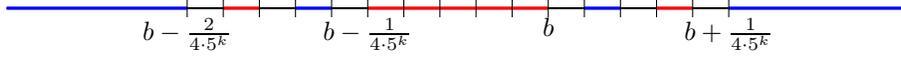
Now we consider the colouring given by $c_{k+1}$. This is shown in Figure~\ref{aroundb2}. The interval $(b- \frac{2}{4 \cdot 5^k}, b+ \frac{1}{4 \cdot 5^k})$ is partitioned into fifteen intervals of width $\frac{1}{4 \cdot 5^{k+1}}$, two of which are coloured blue. Since $d > \frac{1}{4 \cdot 5^{k+1}}$, each of these intervals contains at most one of $\{a, a+d, \dots, a+12d, a+13d \}$. This implies at least least one of the two blue intervals contains a point of $\{a, a+d, \dots, a+12d, a+13d \}$, because you can't fit 14 pigeons into 13 holes without putting two in the same hole. But none of these points are allowed to be blue, a contradiction. This concludes case (1).

For cases (2)-(4), we need the notion of a ladder. We define a ladder to be a finite non-decreasing sequence of real numbers, such that the difference between consecutive terms is at most 1/4, and the difference between the first and last terms is 1. These are useful because, since $[1/2, 3/4]$ is blue in $c_0$, any ladder must contain a blue point of $c_0$.

\textbf{Case (2)} If $d \in [\frac{1}{4}, \frac{5}{16}]$, then $4d-1, 1-3d \in [0, 1/4]$, so 
\[ a, a+4d-1, a+d, a+5d-1, a+2d, a+6d-1, a+3d, a+1 \]
is a ladder for any $a$. $c_0$ cannot contain a $(7,0)$-AP with first term $a$ and $d \in [\frac{1}{4}, \frac{5}{16}]$, because if it did, the above ladder would contain no blue points.

\textbf{Case (3)} If $d \in [\frac{5}{16}, \frac{7}{20}]$, then $5d-3/2, 3/2-4d, 1/2-d, 2d-1/2 \in [0, 1/4]$, so 
\[ a, a+5d-\frac{3}{2}, a+d, a+3d-\frac{1}{2}, a+2d, a+4d-\frac{1}{2}, a+1 \]
is a ladder for any $a$. $c_0$ cannot contain a $(3,3)$-AP with first term $a$ and $d \in [\frac{5}{16}, \frac{7}{20}]$, because if it did, the above ladder would contain no blue points. (Note that $a+3d, a+4d, a+5d$ not being red is the same as $a+3d-\frac{1}{2}, a+4d-\frac{1}{2}$ and $a+5d-\frac{1}{2}$ not being blue, because of the symmetry).

\textbf{Case (4)} If $d \in [\frac{7}{20}, \frac{2}{5}]$, then $3d-1, 2-5d \in [0, 1/4]$, so 
\[ a, a+3d-1, a+6d-2, a+d, a+4d-1, a+7d-2, a+2d, a+5d-1, a+1 \]
is a ladder for any $a$. $c_0$ cannot contain a $(8,0)$-AP with first term $a$ and $d \in [\frac{7}{20}, \frac{2}{5}]$, because if it did, the above ladder would contain no blue points.

\textbf{Case (5)} Let $k \ge 0$ be an integer. We claim $c_k$ does not contain a $(13,0)$-AP with $d \in [\frac{1}{2}-\frac{1}{8 \cdot 5^k}, \frac{1}{2}-\frac{1}{8 \cdot 5^{k+1}})$.
Suppose for contradiction $\{a, a+d, \dots, a+11d, a+12d\}$ is a $(13,0)$-AP

Let $d' = 1/2 - d$ and let $a' = a+12d$. For any integer $t$ with $0 \le t \le 12$, $a' + td' = a + (12-t)d + t/2$. Since $a+ (12-t)d$ is not blue, $a'+td'$ is not blue if $t$ is even and not red if $t$ is odd. We also have that $d' \in (\frac{1}{8 \cdot 5^{k+1}}, \frac{1}{8 \cdot 5^k}]$. Recall $c_k$ consists of the following repeating pattern:
\begin{itemize}
\item a blue interval of length at least $\frac{1}{4 \cdot 5^k}$
\item an uncoloured interval of length $\frac{1}{4 \cdot 5^k}$
\item a red interval of length at least $\frac{1}{4 \cdot 5^k}$
\item an uncoloured interval of length $\frac{1}{4 \cdot 5^k}$
\end{itemize}
Two consecutive points in the arithmetic progression $\{a', a'+d', \dots, a'+11d', a'+12d' \}$ cannot lie in the same coloured interval of $c_k$, since one of these two points is not allowed to be red, and the other is not allowed to be blue. Suppose $a'+kd'$ lies in some coloured interval. If $k \neq 0, 12$, then either $a' + (k-1)d'$ or $a'+(k+1)d'$ must lie in the same coloured interval, giving a contradiction, because the coloured interval has length at least $\frac{1}{4 \cdot 5^k}$ and $d' \le \frac{1}{8 \cdot 5^k}$.

Therefore all of $a'+d', a'+2d', \dots, a'+11d'$ must be in uncoloured intervals. Like in case (1), the common difference is too small to jump over the coloured intervals, so these 11 points must all lie in the same uncoloured interval. But $a'+11d'-(a'+d') = 10d' > \frac{1}{4 \cdot 5^k}$ so $a'+11d'$ and $a'+d'$ are too far from each other to lie in the same uncoloured interval. We have a contradiction. 

This concludes the proof of Theorem~\ref{cinfty}. 
\end{proof}

So we now know that our colouring $c'_\infty$ of $\R/\Z$ does not contain a $(14,14)$-AP, and therefore the same must also be true of the induced colourings of $\Z/p\Z$.

Let $r_p$ and $b_p$ be the proportions of red and blue points respectively, in the colouring on $\Z/p\Z$. The red and blue sets in $c_\infty$ are both countable unions of intervals of total measure 1/2. Since the proportion of points in $\Z/p\Z$ above some interval of length $\ell$ in $\R/\Z$ is $\ell + O(p^{-1})$, this implies that as $p \to \infty$, $\liminf r_p \ge 1/2$ and $\liminf b_p \ge 1/2$. As $r_p + b_p = 1$, This implies $r_p$ and $b_p$ must both tend to $1/2$. So $(14,14)$ is not findable, even for $\delta$ arbitrarily close to 1/2.

\section{Construction for $(3, 30000)$}\label{3vs30000}
\subsection{Fractal construction}
We will use the same kind of construction as in Section~\ref{fc1}, except that the fractal will have a more complicated structure. We aim to construct a colouring $c'_\infty$ of $\R/\Z$ which has no $(3, 30000)$-AP, as this will induce suitable colourings of $\Z/p\Z$. We start by defining an infinite series of configurations $c_0, c_1, c_2, \dots$, each one obtained from the previous one by certain replacement rules. Since the fractal involved is more complicated, we need to use two kinds of uncoloured intervals, each with their own replacement rule.

For the purposes of this section, we define a configuration as follows.
A \emph{configuration} is a partial colouring of $\R/\Z$ together with a partition of the uncoloured subset into intervals, and an assignment of a type and a direction to each of these intervals. Each interval can either be assigned Type I or Type II, independently of its direction, and the directions work in the same way as before. Let $c_0$ be the following configuration:

\begin{itemize}
\item $(0, 1/6)$ Type I uncoloured interval directed backwards
\item $[1/6, 1/3]$ coloured blue
\item $(1/3, 1/2)$ Type I uncoloured interval directed forwards
\item $[1/2,1]$ coloured blue.
\end{itemize}

This is shown in Figure~\ref{c02}. The replacement rules are as follows. Let $(a,b)$ be one of the uncoloured intervals, and $\epsilon$ be its direction. Let $f:[0,1] \to [a,b]$ be given by
\[ f : x \mapsto \begin{cases}
    a + (b-a)x  & \quad \text{if } \epsilon = +1 \\
    b - (b-a)x  & \quad \text{if } \epsilon = -1 
\end{cases} \]
If $(a,b)$ is of Type I, we replace $(a,b) = f((0,1))$ with the following:

\begin{itemize}
\item $f((0, \frac{1}{288}))$ Type II uncoloured, direction $-\epsilon$
\item $f([\frac{1}{288}, \frac{7}{288}])$ coloured red
\item $f((\frac{7}{288}, \frac{1}{36}))$ Type II uncoloured, direction $\epsilon$
\item $f([\frac{1}{36}, \frac{1}{3}])$ coloured blue
\item $f((\frac{1}{3}, 1))$ Type I uncoloured, direction $\epsilon$.
\end{itemize}
Note that the values $\frac{1}{288}$ and $\frac{7}{288}$ are chosen so that the middle $\frac{3}{4}$ of the first $\frac{1}{36}$ of the interval is red.

If $(a,b)$ is of Type II, we replace $(a,b) = f((0,1))$ with the following:

\begin{itemize}
\item $f((0,\frac{1}{101}))$ Type II uncoloured, direction $\epsilon$
\item $f([\frac{1}{101}, \frac{51}{101}])$ coloured blue
\item $f((\frac{51}{101}, 1))$ Type I uncoloured, direction $\epsilon$
\end{itemize}
This is shown in Figure~\ref{replacement2}.
\begin{figure}
\begin{tikzpicture}[scale=2]
\draw[thick] (0,0) -- (6,0)
   node[pos=0, below]   {0}
   node[pos=1/6, below]   {1/6}
   node[pos=1/3, below]   {1/3}
   node[pos=1/2, below]   {1/2}
   node[pos=1, below]   {1};
\draw[very thick, blue] (1, 0) -- (2, 0);
\draw[very thick, blue] (3, 0) -- (6, 0);
\draw[thick, ->] (1, 0) -- (0.5, 0);
\draw[thick, ->] (2, 0) -- (2.5, 0);
\end{tikzpicture}
\caption{The configuration $c_0$. Type I uncoloured intervals are shown in black.}
\label{c02}
\end{figure}
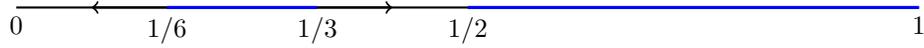
\begin{figure}
\begin{tikzpicture}[scale=1]
  \draw[thick] (0,0) -- (5,0)
  node[pos=0, below]    {0}
  node[pos=1, below]    {1};
  \draw[thick] (9,0) -- (12,0)
  node[pos=0, below]    {0}
  node[pos=0.16, below] {$\frac{1}{36}$}
  node[pos=0.4, below]  {$\frac{1}{3}$}
  node[pos=1, below]    {1};
  \draw[->, thick]  (0,0) -- (2.5,0);
  \draw[->, thick]  (9,0) -- (10.5,0);
  \draw[->, green]  (7.15,0) -- (7.03,0);
  \draw[->, green]  (7.65,0) -- (7.77,0);
  \draw[red, very thick]  (7.15,0)  -- (7.65,0) ;
  \draw[blue, very thick] (7.8,0) -- (9,0);
  \draw[green]      (7,0) -- (7.15, 0);
  \draw[green]      (7.65,0) -- (7.8,0);
  \draw[->] (5.8,0) -- (6,0) node[below] {Becomes} -- (6.2,0);
  \draw[green, thick] (5.4, -2) -- (9.4, -2)
  node[pos=0, black, below] {0}
  node[pos=0.125, black, below] {$\frac{1}{288}$}
  node[pos=0.875, black, below] {$\frac{7}{288}$}
  node[pos=1, black, below] {$\frac{1}{36}$};
  \draw[red, very thick]   (5.9, -2) -- (8.9, -2);
  \draw[->, green, thick]  (5.9, -2) -- (5.6, -2);
  \draw[->, green, thick]  (8.9, -2) -- (9.2, -2);
  \draw (5.4, -1.8) -- (6.85, -0.2);
  \draw (9.4, -1.8) -- (8, -0.2);
\end{tikzpicture}
\begin{tikzpicture}[scale=1]
  \draw[thick, green] (0,0) -- (5,0)
  node[pos=0, black, below]    {0}
  node[pos=1, black, below]    {1};
  \draw[thick] (7,0) -- (12,0)
  node[pos=0, below]    {0}
  node[pos=0.12, below] {$\frac{1}{101}$}
  node[pos=0.56, below]  {$\frac{51}{101}$}
  node[pos=1, below]    {1};
  \draw[->, thick, green]  (0,0) -- (2.5,0);
  \draw[->, thick,green]  (7,0) -- (7.3,0);
  \draw[->, thick]  (9.8,0) -- (10.9,0);
  \draw[green, thick]  (7,0)  -- (7.6,0) ;
  \draw[blue, very thick] (7.6,0) -- (9.8,0);
  \draw[->] (5.8,0) -- (6,0) node[below] {Becomes} -- (6.2,0);
\end{tikzpicture}
\caption{The replacement rules. Type I uncoloured intervals are shown in black and Type II uncoloured intervals are shown in green. The drawings are not to scale. For the first replacement rule, we zoom in on the first $\frac{1}{36}$ of the interval to show detail. These are shown for the interval $(0,1)$. The rules for a general interval are obtained by applying the linear function $f$.}
\label{replacement2} 
\end{figure}
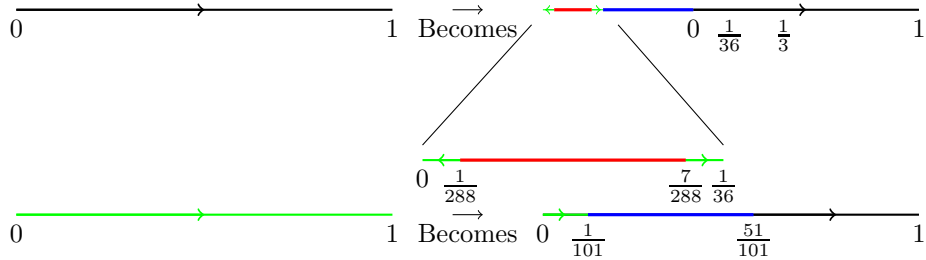
In the limit, we will get a partial colouring $c_\infty$ of $\R/\Z$. We then form our colouring $c'_\infty$ by arbitrarily colouring the uncoloured elements from $c_\infty$. 
\subsection{Proof there is no $(3,30000)$-AP}
As in the previous section, we will use the convention that, when dealing with a partial red-blue colouring, we define an $(a,b)$-AP to be an arithmetic progression of length $a+b$ such that none of the first $a$ elements are blue and none of the last $b$ elements are red. It suffices to prove that $c_\infty$ contains no non-trivial $(3,30000)$-AP under this definition.

First, we establish the following properties of the configurations $c_k$:
\begin{enumerate}
\item 
\begin{enumerate}
\item Every Type I uncoloured interval of length $\ell$ in $c_k$ is preceded by a blue interval of length at least $\frac{11}{24} \ell$
\item Every Type I uncoloured interval of length $\ell$ in $c_k$ is succeeded by a blue interval of length at least $3 \ell$
\end{enumerate}
\item 
\begin{enumerate}
\item Every Type II uncoloured interval of length $\ell$ in $c_k$ is preceded by a red interval of length at least $6 \ell$
\item Every Type II uncoloured interval of length $\ell$ in $c_k$ is succeeded by a blue interval of length at least $6 \ell$
\end{enumerate}
\item Behind a Type I interval of length $\ell$ in $c_k$, within a distance of $\frac{25}{24} \ell$, we can find a red interval of $c_{k+1}$ of length at least $\frac{\ell}{48}$
\item If a Type II interval of length $\ell$ in $c_k$ came from a Type I interval of $c_{k-1}$, then in front of that Type II interval, within a distance of $300 \ell$, we can find a red interval of $c_{k+1}$ of length at least $4 \ell$
\end{enumerate}
Note that the meanings of the terms ``preceded'', ``succeeded'', ``in front of'' and ``behind'' are determined by the direction of the uncoloured interval. We will prove all of these properties together, by induction on $k$.
\begin{proof}
We start of with the base case $k=0$. Since there are no Type II uncoloured intervals in $c_0$, we only need to check properties 1 and 3. $c_0$ has two uncoloured intervals of Type I, and each has length $\frac{1}{6}$. They are $(0,\frac{1}{6})$ directed backwards and $(\frac{1}{3}, \frac{1}{2})$ directed forwards. They are each preceded by a blue interval of length $\frac{1}{6}$ and succeeded by a blue interval of length $\frac{1}{2}$, so Property 1 holds.

Since $c_0$ is symmetric about $\frac{1}{4}$, it suffices to check Property 3 holds for $(\frac{1}{3}, \frac{1}{2})$. When we apply the replacement rule to $(0, \frac{1}{6})$, we get a red interval $(\frac{281}{1728}, \frac{287}{1728})$ of length $\frac{1}{288} = \frac{1}{48} \times \frac{1}{6}$. This interval is in $c_1$ and is within $\frac{25}{24} \times \frac{1}{6}$ of $\frac{1}{3}$ (as $\frac{1}{3} - \frac{25}{24} \times \frac{1}{6} = \frac{23}{144} = \frac{276}{1728}$).

Now for the induction step. For $k>0$, every uncoloured interval of $c_k$ must arise from applying a replacement rule to an uncoloured interval of $c_{k-1}$, so we work case by case.

\underline{Property 1}

First case: an uncoloured interval of Type I that came from an uncoloured interval of Type I and length $\ell_0$ in $c_{k-1}$. The interval has length $\ell = \frac{2}{3} \ell_0$, and is preceded by a blue interval of length $\frac{11}{36} \ell_0 = \frac{11}{24} \ell$. It is succeeded by a blue interval of length at least $3 \ell_0 > 3 \ell$ (using Property 1(b) for $c_{k-1}$), so Property 1 holds.

Second case: an uncoloured interval of Type I that came from an uncoloured interval of Type II and length $\ell_0$ in $c_{k-1}$. The interval has length $\ell = \frac{50}{101} \ell_0$, and is preceded by a blue interval of length $\frac{50}{101} \ell_0 = \ell > \frac{11}{24} \ell$. It is succeeded by a blue interval of length at least $6 \ell_0 > 3 \ell$ (using Property 2(b) for $c_{k-1}$), so Property 1 holds.

\underline{Property 2}

First case: an uncoloured interval of Type II that came from an uncoloured interval of Type I and length $\ell_0$ in $c_{k-1}$. The interval has length $\ell = \frac{1}{288} \ell_0$, and is preceded by a red interval of length $\frac{1}{48} \ell_0 = 6 \ell$. We have two subcases, either it is ``pointing inwards'', in which case it is succeeded by a blue interval of length $\frac{11}{36} \ell_0 > 6 \ell$, or it is ``pointing outwards'' in which case it is succeeded by a blue interval of length at least $\frac{11}{24} \ell_0 > 6 \ell$ (using Property 1(a) for $c_{k-1}$). So Property 2 holds.

Second case: an uncoloured interval of Type II that came from an uncoloured interval of Type II and length $\ell_0$ in $c_{k-1}$. The interval has length $\ell = \frac{1}{101} \ell_0$, and is preceded by a red interval of length at least $6 \ell_0 > 6 \ell$ (using Property 2(a) for $c_{k-1}$). It is succeeded by a blue interval of length $\frac{50}{101} \ell_0 > 6 \ell$, so Property 2 holds.

\underline{Property 3}

First case: an uncoloured interval of Type I that came from an uncoloured interval of Type I and length $\ell_0$ in $c_{k-1}$. The interval has length $\ell = \frac{2}{3} \ell_0$, and we have a red interval of length $\frac{1}{48} \ell_0 = \frac{1}{32} \ell$ also arising from the same Type I uncoloured interval of $c_{k-1}$. It is within $\frac{1}{3}\ell_0 = \frac{1}{2}\ell$ of the rear end of our interval, so Property 3 holds.

Second case: an uncoloured interval of Type I that came from an uncoloured interval of Type II and length $\ell_0$ in $c_{k-1}$. Let $(m, m+\ell_0)$ be this Type II interval of $c_{k-1}$, which we may assume without loss of generality is directed forwards. Our Type I interval is $(m+\frac{51}{101} \ell_0, m+\ell_0)$. It has length $\ell = \frac{50}{101} \ell_0$. By aplying property 2(a) in $c_{k-1}$, we can see that $[m-\frac{1}{96} \ell_0, m]$ is red. This has length $\frac{1}{96} \ell_0 = \frac{101}{4800} \ell > \frac{1}{48} \ell$, and is within $(\frac{51}{101}+\frac{1}{96})\ell_0 = (\frac{51}{50}+\frac{101}{4800})\ell = \frac{4997}{4800} \ell < \frac{25}{24} \ell$, so Property 3 holds.

\underline{Property 4}

Consider an uncoloured interval of Type II that came from an uncoloured interval of Type I and length $\ell_0$ in $c_{k-1}$.  The interval has length $\ell = \frac{1}{288} \ell_0$. Here we have two subcases, the ``pointing outwards'' and ``pointing inwards'' cases. In the ``pointing outwards'' case, we use property 3 in $c_{k-1}$ to deduce that there is a red interval of length at least $\frac{1}{48} \ell_0 = 6 \ell$ within $\frac{25}{24} \ell_0 = 300 \ell$. In the ``pointing inwards'' case, we note that there is a Type I interval of length $\frac{2}{3} \ell_0$ that came from the same Type I interval of $c_{k-1}$. In $c_{k+1}$, this will give us a red interval of length $\frac{2}{3} \times \frac{1}{48} \ell_0 = \frac{1}{72} \ell_0 = 4 \ell$. It is within a distance of $\ell_0 = 288 \ell$ in front of our Type II interval, so Property 4 holds.

This completes the induction step and concludes our proof.
\end{proof}

Now we can proceed with the proof that $c_\infty$ contains no non-trivial $(3,30000)$-AP. Suppose, for a contradiction, that it does contain one, and it has first term $a$ and common difference $d \neq 0$. Because the initial configuration $c_0$ is symmetric (reflection in $1/4$), it follows that $c_i$ is also symmetric for all $i \le \infty$. Therefore we may assume without loss of generality that $d \in (0, 1/2]$. Also, the $(3,30000)$-AP is a $(3,30000)$-AP in $c_i$ for all $i$, because the partial colouring in each $c_i$ is a subcolouring of $c_\infty$.

We start off by considering just the first three points of the arithmetic progression, and where they might lie. None of these three points are allowed to be blue.

Firstly, all three of them must lie in the same Type I uncoloured interval of $c_0$. This is because the not-blue set in $c_0$ is $(0,1/6) \cup (1/3,1/2)$, and the only way you can have an arithmetic progression of length three in $\R/\Z$ with all three points lying in $(0,1/6) \cup (1/3,1/2)$ is to either have all three of them lying in $(0,1/6)$ or all three lying in $(1/2, 1/3)$.

So let $k$ be maximal such that $c_k$ contains a Type I uncoloured interval which contains all three points. Call this interval $I_1$, and say it starts at $s$ and ends at $s+t$ ($t$ will be negative if $I_1$ is directed backwards). Let the first $\frac{1}{36}$ of $I_1$ be $J_1$, and the last $\frac{2}{3}$ be $I_2$. In $c_{k+1}$, $I_2$ is a Type I uncoloured interval, directed in the same direction as $I_1$, and $J_1 \cup I_2$ is the subset of $I_1$ which isn't blue. We similarly define $J_2$ to be the first $\frac{1}{36}$ of $I_2$, $I_3$ to be the last $\frac{2}{3}$ of $I_2$, $J_3$ to be the first $\frac{1}{36}$ of $I_3$ and $I_4$ to be the last $\frac{2}{3}$ of $I_3$. $J_1 \cup J_2 \cup I_3$ contains all the not blue points of $c_{k+2}$ from $I_1$, while $J_1 \cup J_2 \cup J_3 \cup I_4$ contains all the not blue points of $c_{k+3}$ from $I_1$. So all three points must lie in $J_1 \cup J_2 \cup J_3 \cup I_4$. This is all shown in Figure~\ref{IsandJs}. The three terms of the arithmetic progression can't all lie in $I_2$ by the maximality of $k$, so at least one must lie in $J_1$. We claim that in fact all three of them lie in $J_1$.

\begin{figure}
\begin{tikzpicture}[scale=0.95]
\draw (0,3) -- (12, 3)
  node[pos = 0, left] {$c_k$}
  node[pos = 0.5, below] {$I_1$}
  node[pos = 0, above] {$s$}
  node[pos = 1, above] {$s+t$};
\draw[->] (0,3) -- (6,3);

\draw (0,2) -- (1.5,2)
  node[pos = 0, left] {$c_{k+1}$}
  node[pos = 0.5, below] {$J_1$}
  node[pos = 1, above] {$s+\frac{1}{36}t$};
\draw[thick, blue] (1.5,2) -- (4,2);
\draw (4,2) -- (12,2)
  node[pos = 0.5, below] {$I_2$}
  node[pos = 0, above] {$s+\frac{1}{3}t$};
\draw[->] (4,2) -- (8,2);

\draw (0,1) -- (1.5,1)
  node[pos = 0, left] {$c_{k+2}$}
  node[pos = 0.5, below] {$J_1$};
\draw[thick, blue] (1.5,1) -- (4,1);
\draw (4,1) -- (5,1)
  node[pos = 0.5, below] {$J_2$}
  node[pos = 1, above] {$s+\frac{19}{54}t$};
\draw[thick, blue] (5,1) -- (7,1);
\draw (7,1) -- (12,1)
  node[pos = 0.5, below] {$I_3$}
  node[pos = 0, above] {$s+\frac{5}{9}t$};
\draw[->] (7,1) -- (9.5,1);

\draw (0,0) -- (1.5,0)
  node[pos = 0, left] {$c_{k+3}$}
  node[pos = 0.5, below] {$J_1$};
\draw[thick, blue] (1.5,0) -- (4,0);
\draw (4,0) -- (5,0)
  node[pos = 0.5, below] {$J_2$};
\draw[thick, blue] (5,0) -- (7,0);
\draw (7,0) -- (7.7,0)
  node[pos = 0.5, below] {$J_3$}
  node[pos = 1, above] {$s+\frac{46}{81}t$};
\draw[thick, blue] (7.7,0) -- (9,0);
\draw (9,0) -- (12,0)
  node[pos = 0.5, below] {$I_4$}
  node[pos = 0, above] {$s+\frac{19}{27}t$};
\draw[->] (9,0) -- (10.5,0);

\draw[dashed] (0,0) -- (0,3);
\draw[dashed] (12,0) -- (12,3);
\draw[dashed] (1.5,0) -- (1.5,2);
\draw[dashed] (4,0) -- (4,2);
\draw[dashed] (5,0) -- (5,1);
\draw[dashed] (7,0) -- (7,1);
\draw (7.7,0) -- (7.7,0.15);
\draw (9,0) -- (9,0.15);
\end{tikzpicture}
\caption{Configurations $c_k$ through $c_{k+3}$, between $s$ and $s+t$. The intervals $I_i$ are Type I uncoloured in each case, while the $J_i$ contain a variety of intervals. We argue that if the first three points are all in $I_1$ but not all in $I_2$ then they are all in $J_1$.}
\label{IsandJs}
\end{figure}
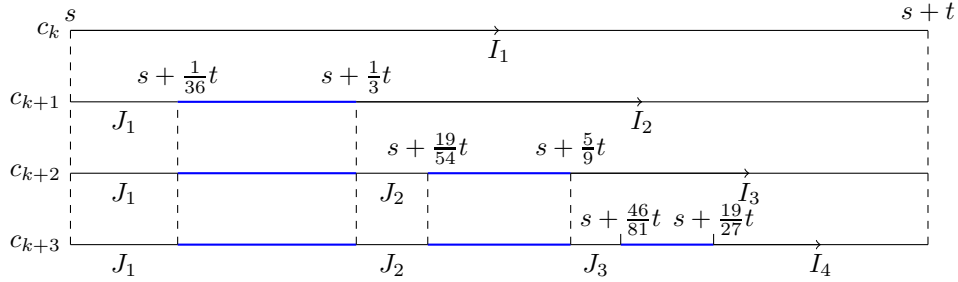

Say the three points are $p_1 = s+u_1t, p_2 = s+u_2t$ and $ p_3 = s+u_3t$, with $u_1, u_2, u_3 \in (0,1)$ in increasing order and $2u_2 = u_1+u_3$.  $p_1$ must lie in $J_1$, so $u_1 \in (0,\frac{1}{36})$ . If $p_2 \in J_1$ as well, then we must have $p_3 \in J_1$ since the gap between $J_1$ and $I_2$ is wider than $J_1$ is. Then all three points lie in $J_1$ as required.

There are two other options: $p_2$ can lie in $J_2$ or $I_3$. $I_3$ Goes form $s+\frac{5}{9}t$ to $s+t$, so in the latter case $u_2 > \frac{5}{9}$ and $u_1 < \frac{1}{36}$ so $u_3 > \frac{10}{9} - \frac{1}{36} > 1$ which can't happen. $J_2$ goes from $s+\frac{1}{3}t$ to 
$s+\frac{19}{54}t$, so if $p_2$ lies in $J_2$, then $u_3 \in (\frac{2}{3}-\frac{1}{36}, \frac{19}{27})= (\frac{23}{36}, \frac{19}{27})$. However, $J_3$ ends at $s + \frac{46}{81}t$ and we have $\frac{46}{81} < \frac{23}{36}$ while $I_4$ starts at $s+\frac{19}{27}t$, which shows that in this case the point $p_3$ will lie in between $J_3$ and $I_4$, which is not allowed.

So all three points lie in $J_1$. In $c_{k+1}$, the middle $\frac{3}{4}$ of $J_1$ is red, while the first $\frac{1}{8}$ and last $\frac{1}{8}$ are Type II uncoloured intervals. Let $J'_1$ be the middle $\frac{3}{4}$ of $J_1$. We claim that at least one of the three points lies in $J_1$.

Suppose this is not the case. One of the Type II uncoloured intervals must contain all three points, because the arithmetic progression cannot jump over $J'_1$. Let $m$ be maximal such that all three points lie in the same Type II uncoloured interval of $c_m$. Let $K_1$ be this Type II interval. We have $m \ge k+1$. In $c_{m+1}$, we cannot have all three points  lying in the same Type II interval, or in the same Type I interval, by the maximality of $k$ and $m$. In $c_{m+1}$, $K_1$ is replaced by a Type II uncoloured interval, followed by a blue interval, followed by a type I uncoloured interval, as shown in Figure~\ref{replacement2}. The three points are not allowed to be blue. If two of the three points lie in the Type II uncoloured interval of $c_{m+1}$, then the third one cannot lie in the Type I uncoloured interval of $c_{m+1}$, and vice versa, because the gap is too big. But we also cannot have all three points lying in the Type I interval, or all three in the Type II interval. This has exhausted all the possibilities, and so we have reached a contradiction. At least one of the three points must lie in $J'_1$.

Let $r$ be the length of the interval $J_1$. Since the first three points of our $(3,30000)$-AP are contained in $J_1$, we have that $d \le \frac{r}{2}$. To conclude our proof, we will apply the following lemma a couple of times.

\begin{lemma}\label{structureoftheargument}
Suppose we know that $d \le d_{\max} $ for some $d_{\max}$, and that in front of the interval $J'_1$ in the positive direction there is at some distance a blue interval $L_1$ of length at least $d_{\max}$, and at some further distance there is a red interval $L_2$ of length at least $d_{\max}$. Let $\ell$ be the distance from the end of $J'_1$ to the start of $L_2$. Then we have
\[ d \le \frac{\ell}{29999} \]
\end{lemma}
\begin{proof}
If our $(3,30000)$-AP contains any points after the start of $L_2$, then it must contain at least one point of the interval $L_2$, and also at least one point of the interval $L_1$, because both of these intervals are too long to jump over. So the $(3,30000)$-AP contains a red point in $J'_1$, a blue point in $L_1$ and a red point in $L_2$, and they occur in that order. But this cannot happen. Therefore all of the points must occur before the start of $L_2$. Also, the last 30000 points aren't allowed to be red, so they must occur after the start of $J'_1$. Therefore we have 30000 points of the arithmetic progression occurring within an interval of length $\ell$, and the conclusion must hold.
\end{proof}

Since $d \le \frac{r}{2}$, we start off by applying Lemma~\ref{structureoftheargument} with $d_{\max} = \frac{r}{2}$. The last $\frac{1}{8}$ of $J_1$ is a Type II interval of $c_{k+1}$ that came from a Type I interval of $c_k$. It has length $\frac{r}{8}$. By applying Property~2(b), we know that it is followed by a blue interval of length at least $\frac{3r}{4}>d_{\max}$. By applying property 4, we also know that within a distance of $\frac{300r}{8}$, there is a red interval of length at least $\frac{r}{2} = d_{\max}$. We apply Lemma~\ref{structureoftheargument} using these intervals, and deduce that $d \le \frac{300 r}{8 \cdot 29999} \le \frac{r}{799}$.

We now claim that, for any integer $n\ge 0$, if $d \le \frac{r}{799(101)^n}$, then $d \le \frac{r}{799(101)^{n+1}}$. We do this by applying Lemma~\ref{structureoftheargument} again, using $d_{\max} = \frac{r}{799(101)^n}$. Let $b$ be the endpoint of $J'_1$. Start off by considering the Type II interval of $c_{k+1}$ which starts at $b$. When we pass to $c_{k+n+1}$, the first $\frac{1}{101^n}$ of this original interval will be a Type II interval oriented forwards. This interval is $(b, b+\frac{r}{8(101)^n}$). We then consider what happens in this interval in the configuration $c_{k+n+3}$ (i.e. after applying the replacement rules two more times). We will have, among other things:
\begin{itemize}
\item A blue interval $(b+\frac{r}{8(101)^{n+1}}, b+\frac{51r}{8(101)^{n+1}})$, whose length $\frac{50r}{8(101)^{n+1}}$ is greater than $d_{\max}$
\item A red interval $(b+\frac{51r}{8(101)^{n+1}}+ \frac{50r}{8 \cdot 288(101)^{n+1}}, b+\frac{51r}{8(101)^{n+1}}+\frac{350r}{8 \cdot 288(101)^{n+1}})$, whose length $\frac{50r}{8 \cdot 48(101)^{n+1}}$ is also greater than $d_{\max}$
\end{itemize}
Applying Lemma~\ref{structureoftheargument} with these intervals, we get that $d \le \frac{1}{29999}(\frac{51r}{8(101)^{n+1}}+ \frac{50r}{8 \cdot 288(101)^{n+1}}) \le \frac{52}{8\cdot 29999}(\frac{r}{101^{n+1}}) \le \frac{r}{799(101)^{n+1}}$ as required.

Since $d \le \frac{r}{799(101^{n})}$ implies $d \le \frac{r}{799(101)^{n+1}}$ for all integers $n \ge 0$, and we already know $d \le \frac{r}{799}$, we must have that $d = 0$ contradicting the assumption that the $(3,30000)$-AP was non-trivial. This concludes our proof that there is no $(3,30000)$-AP in $c'_\infty$.

We only need to check the densities now. Let $r_p$ and $b_p$ be the densities of red and blue points in the colourings of $\Z/p\Z$ induced from $c'_\infty$. One can show that in $c_\infty$, the red set is a countable union of intervals of measure $\frac{2}{95}$ and the blue set is a countable union of intervals of measure $\frac{93}{95}$. So, by the same reasoning as in the previous section, $r_p$ and $b_p$ tend to $\frac{2}{95}$ and $\frac{93}{95}$ respectively. Therefore, we contradict the findability of $(3,30000)$ for all $\delta < \frac{2}{95}$.

We end by remarking that, unlike our construction for $(14,14)$-APs, this construction does not have density close to $\frac{1}{2}$, so it is an open question whether there exists a $k$ such that $(3,k)$ is not findable in colourings of density close to $\frac{1}{2}$.

\emph{Acknowledgements.} I would like to thank Ben Green for some helpful comments.

\bibliographystyle{plain}
\bibliography{Bibliography}

\end{document}